\documentclass[a4paper, 11pt]{article}

\usepackage{enteteanglais}

\title{Minimal entropy for uniform lattices in $\left(\PSL_2(\R)\right)^n$}

\author{Louis Merlin}

\begin{document}

\maketitle

\begin{abstract}
We prove that, among metrics on a compact quotient of $\mathbb{H}^2 \times\cdots\times \mathbb{H}^2$ (product of hyperbolic planes) of prescribed total volume, the product of hyperbolic metrics has minimal volume entropy.
\end{abstract}


\section{Introduction}

Let $(X,g)$ be a compact Riemannian $n$-manifold and $\widetilde{X}$ be the universal Riemannian cover of $X$. The volume entropy is defined as
\[h(g)=\lim_{R \to \infty}{\frac{1}{R}\log(\Vol(B(x,R)))}\]
where $B(x,R)$ is the ball of radius $R$ in $\widetilde{X}$ centered at any point $x\in\widetilde{X}$. The limit exists and is independent of the choice of $x$ (see \cite{manning} p.568).

In this paper, we are interested in the following problem.

\begin{quest} \label{question}
Let $M$ be a compact locally symmetric space of noncompact type with locally symmetric metric $g_0$. Let $g$ be any other metric on $M$ such that $\Vol(M,g_0)=\Vol(M,g)$. Do we have
\[h(g) \geqslant h(g_0)?\]
\end{quest}

M. Gromov was the first to conjecture such a result in \cite{gromovfilling}. He was only interested in the real hyperbolic case. But the question still makes sense for a general symmetric space.\\

In the case where $\widetilde{M}$ is reducible, there exist a unique locally symmetric metric of minimal entropy among locally symmetric metrics of prescribed volume, which are obtained by scaling the metric in the factors (\cite{connellfarbmer} Chapter 2). This metric is called "the" locally symmetric metric and is denoted by $g_0$.\\

In this work, we give a positive answer to the previous question in the case of compact quotients of products of $n$ hyperbolic planes, that is $M=\Gamma\backslash\left(\mathbb{H}^2\right)^n$ where $\Gamma$ is a uniform lattice in $\left(\PSL_2(\R)\right)^n$. More precisely our main result is the following:

\begin{thm}[Main theorem] \label{main}
Assume that $(M,g_0)$ is a compact quotient of the product of $n$ hyperbolic planes. Then, for any other metric $g$ on $M$,
\[h^{2n}(g)\Vol(M,g) \geqslant h^{2n}(g_0)\Vol(M,g_0).\]
\end{thm}

Let us remark that the above inequality is sharp and no assumption is made on the metric $g$.\\

According to \cite{eberlein} p.260, there exist in PSL$_2(\mathbb{R}) \times $PSL$_2(\mathbb{R})$ two types of uniform lattices: virtual products and irreducible lattices.
Then our main theorem applies for the compact products of hyperbolic surfaces or Riemannian manifolds finitely covered by such a product but also for quotients of $\left(\mathbb{H}^2\right)^n$ which are far from being products. Both compact examples do exist (see \cite{shimizu} chapter 6 or \cite{borel} for arithmetic examples).\\

In the fundamental paper \cite{bcg2}, G. Besson, G. Courtois and S. Gallot dealt with the case where $M$ is a locally symmetric space of rank one ($g_0$ is negatively curved) and obtained a similar statement than the main theorem for such spaces. The same result was obtained before (see \cite{bcg1}) in any rank but for a metric $g$ in the conformal class of the locally symmetric metric $g_0$ (supposed to be irreducible). In the case where $\dim M \geqslant 3$, the method is based on the barycenter map and the inequality in theorem \ref{main} appears as an inequality of calibration (to be described below). In the case where $\dim M=2$, already proved by Katok (see \cite{katokconformal}), one can still ask if the inequality can be seen as an inequality of calibration. Besson, Courtois and Gallot showed that it is indeed true and gave another proof of the conjecture for the hyperbolic surfaces.

The barycenter method was improved by Connell and Farb in \cite{connellfarbmer} and, working factor by factor, they also gave a positive answer to the question \ref{question} in the case where $M$ is locally a product of rank one symmetric spaces with no factor $\mathbb{H}^2$. In fact in that case, there exist a unique locally symmetric metric (when one scales the metric in each factor) of minimal entropy and with volume one.

In both of those papers, the authors pointed out that the case of products of hyperbolic surfaces still remained unknown.

Note also that there is an answer to question \ref{question} in the same setting as \cite{connellfarbmer} (products of rank $1$ symmetric spaces without $\mathbb{H}^2$-factors) in \cite{bcgmilnorwood} using an interesting different point of view. The inequality between volumes and entropies appears as a corollary of a general work on representations of fundamental groups of compact manifolds into Lie groups of noncompact types.\\

We now describe the content of this paper. To prove our main result, we use the general outline introduced in \cite{bcg2}. It consists on an application of a method of calibration. To make this method efficient, we embed the universal cover $\mathbb{H}^2 \times\cdots\times \mathbb{H}^2$ in the unit sphere of $\Le^{2}$-functions of the Furstenberg boundary $\mathbb{T}^n=\mathbb{S}^1 \times\cdots\times \mathbb{S}^1$ by the products of Poisson kernels. The aim is to show that this embedding has minimal volume. In order to detect this minimality property, one may use a differential $2n$-form taking its extremal values over orthonormal frames in tangent frames of the embedding. The method is briefly recalled in sections \ref{spherevol} and \ref{calibration}.

The hardest part is to find the calibrating form. Apart from the barycenter map, which is not efficient in the $2$-dimensional case, Besson, Courtois and Gallot developed an alternative idea.
In Chapter 3 of \cite{bcg1}, following Gromov \cite{gromovvabc}, there is a general process to build suitable differential forms using bounded cocycles. The choice of the appropriate cocycle in $\mathbb{H}^2$ is then discussed in \cite{bcg2} Chapter 6.
We generalize this approach for the compact quotients of $\left(\mathbb{H}^2\right)^n$ using the bounded $2n$-cocycle on $\mathbb{T}^n$ that M. Bucher exploits in \cite{bucherh22} (for $n=2$). We describe this bounded cocycle in paragraph \ref{form} and we check that the derived differential form has the required properties afterwards. The calibrating inequality is finally obtained in paragraph \ref{inequality}.

The last section is devoted to applications. We obtain a (non optimal) estimate for the minimal volume of a compact quotient $\Gamma\backslash\left(\mathbb{H}^2\right)^n$. The most spectacular application is that we are able to give an optimal bound for degrees of maps
$f : Y^{2n}\longrightarrow \Gamma\backslash\left(\mathbb{H}^2\right)^n$ from any Riemannian $2n$-manifold. Precisely,

\begin{cor}
Let $Y$ be a smooth manifold of dimension $2n$ endowed with a Riemannian metric $g$ and let $f$ be a continuous map
\[f : (Y,g) \longrightarrow (M,g_0)\]
Then
\[h(g)^{2n} \Vol(Y,g) \geqslant \abs{\textit{deg }f} h(g_0)^{2n} \Vol(M,g_0)\]
\end{cor}

I would like thank my PhD advisor C. Bavard for many useful discussions and comments. I am also grateful to G. Besson for his encouragements and the remarks he made on a preliminary version of this paper.

\section{Calibration method}\label{method}

\subsection{The spherical volume}\label{spherevol}

In this section, $M$ is a quotient of $\left(\mathbb{H}^2\right)^n$ by a uniform torsion-free lattice $\Gamma=\pi_1(M)$ and then $\widetilde{M}$ is the Riemannian product $\left(\mathbb{H}^2\right)^n$ with the usual metric of curvature $-1$ in each factors (called $g_0$ in both of the manifolds $M$ and $\tilde{M}$). Remark that $g_0$ is the best locally symmetric metric in the sense of \cite{connellfarbmer} Chapter 2. As above, $g$ is any other metric on $M$.\\

Let us start with a few notations. We choose once and for all a basepoint $o \in \mathbb{H}^2$, for instance $o=(0,0)$ in the Poincaré disk model. We will denote by the same letter $o$ the basepoint in $\left(\mathbb{H}^2\right)^n$. There will be no ambiguity resulting of that convention. The basepoint is used to compute the Buseman functions and to identify the boundary at infinity $\pa_{\infty} \mathbb{H}^2$ with the circle $\mathbb{S}^1$ (see below).

In what follows, $\pa_F(\widetilde{M})$ will denote the Furstenberg boundary of $\widetilde{M}$: the space of Weyl chambers at infinity in $\widetilde{M}$ emanating from the same point (see \cite{eberlein} or \cite{gjt} for further discussions).
There will be no conceptual difficulties coming from a general theory of Furstenberg boundary: we just use the fact that $\pa_F(\left(\mathbb{H}^2\right)^n)$ is identified to the $n$-dimensional torus $\mathbb{T}^n=\left(\mathbb{S}^1\right)^n$, by the above choice of a basepoint.
The Furstenberg boundary is better adapted to the case of higher rank symmetric spaces and it is one of the key points in \cite{connellfarbmer}. The Furstenberg boundary and the visual boundary are the same in the rank one case and that's why the distinction doesn't appear in \cite{bcg2}.\\

The Furstenberg boundary turns out to be a probability space in the following way. The circle $\mathbb{S}^1=\R\slash 2\pi\Z$ is endowed with the Lebesgue probability measure $d\theta$ (normalized in such a way that $d\theta(\mathbb{S}^1)=1$). The $n$-torus is the product (in the sense of probability spaces) of $n$ such circles. The spaces of $\Le^2$ functions on $\mathbb{T}^n$ is defined with respect to this measure.\\

The Poisson kernel $p_o$ of the disk is defined by:
\[p_o(x,\theta)=e^{-\mathcal{B}_o(x,\theta)},\]
where $\mathcal{B}_o(x,\theta)$ is the Buseman function.
A classical computation gives the explicit expression in the Poincaré disk
\[p_o(x,\theta)= \frac{1-\abs{x}^2}{\abs{x-e^{2i\pi\theta}}^2}\]
for any $x \in B(0,1)$ and $\theta \in \mathbb{S}^1$.\\

The definition of the spherical volume in \cite{bcg2} extends to $\left(\mathbb{H}^2\right)^n$ in the following way. We consider two representations of $\Gamma$. The first one in Isom$(\tilde{M},g_0)$ is the holonomy representation of $\Gamma$. The second one is the representation in the unit sphere of $\Le^2(\partial_F\widetilde{M})$, the Hilbert space of $\Le^{2}$-functions on the Furstenberg boundary with real values. We denote this unit sphere by $\mathcal{S}^{\infty}$ or $\mathcal{S}^{\infty}(\pa_F\widetilde{M})$ if the universal cover needs to be specified. More precisely, it is a unitary representation restricted to $\mathcal{S}^{\infty}$. It is defined by
\[(\gamma f) (\theta)= f(\gamma^{-1}(\theta))\sqrt{p_0(\gamma o_1,\theta^1)}\cdots \sqrt{p_0(\gamma o_n,\theta^n)} \]
where $o$ is the basepoint of $\left(\mathbb{H}^2\right)^n$, $\gamma o =(\gamma o_1,\cdots,\gamma o_n) \in \left(\mathbb{H}^2\right)^n$ and $\theta=(\theta^1,\cdots,\theta^n) \in \mathbb{T}^n$. This is the change of variables formula (for $\Le^2$ functions). Indeed the Jacobian of an isometry acting on the Furstenberg boundary is given by the product of Poisson kernels.\

Then, as in \cite{bcg2}, we introduce the family $\mathcal{N}$ of Lipschitz immersions
\[\Phi: \widetilde{M} \longrightarrow \Le^2(\pa_F\widetilde{M})\]
which are $\Gamma$-equivariant, that is satisfying the following equation
\[\Phi(\gamma x)=\gamma \Phi(x)\]
for all $\gamma\in\Gamma$ and all $x\in \widetilde{M}$. 
We also require that
\[\forall x \in \left(\mathbb{H}^2\right)^n,\;\;\; \norm{\Phi(x)}_{\Le^2}=1, \]
that is $\Phi(x) \in \mathcal{S}^\infty$ and that, for every $x\in\widetilde{M}$, $\Phi(x)$ is positive almost everywhere.
We can also consider those immersions as functions of two variables $\Phi : (x,\theta) \mapsto \Phi(x)(\theta)$.
The product of square roots of Poisson kernels
\[
\begin{array}{c c c c}
\Phi_0 = : & \left(\mathbb{H}^2\right)^n & \longrightarrow & \rons^{\infty} \\
& (x^1,\cdots,x^n) & \longmapsto & \sqrt{p_o(x^1,\cdot)}\times\cdots\times \sqrt{p_o(x^n,\cdot)}\\
\end{array}
\]
is an example of such an immersion (see Lemma \ref{poisson} below). Moreover it is an embedding. We will think of $\left(\mathbb{H}^2\right)^n$ as embedded in $\mathcal{S}^{\infty}$ by the product of Poisson kernels.
The spherical volume is then defined by
\[
\SphereVol(M)  =  \inf_{\Phi \in \mathcal{N}} \{\Vol (\Phi )\}\\
\]
where
\[
\Vol(\Phi)= \int_M{\sqrt{\abs{\det\,_{g_0}(g_{\Phi}(x))}} dv_{g_0}(x)} \}
,\]
where $g_{\Phi}$ denotes the almost everywhere defined pull-back of the usual Hilbertian metric on $\Le^2$ by the Lipschitz immersion $\Phi$ and $\det_g(g_{\Phi})$ is computed in any $g$-orthonormal basis. The integral on $M$ means that we integrate on a fundamental domain in $\left(\mathbb{H}^2\right)^n$ for the $\Gamma$-action and the equivariance relation satisfied by $\Phi$ shows that this does not depend on the fundamental domain.

The spherical volume is a transitional object. We use it to make a link between entropies and volumes. We want to prove the following inequalities.
\[\Vol(M,g_0) \left(\frac{h(g_0)^2}{8n}\right)^{n}=\SphereVol(M) \leqslant \Vol(M,g) \left(\frac{h(g)^2}{8n}\right)^{n}.\]

First we recall the second inequality.
\begin{prop}[\cite{bcg2} Chapter 3]\label{boundedabove}
We have
\[\SphereVol(M) \leqslant \left(\frac{h(g)^2}{8n}\right)^{n} \Vol (M,g).\]
\end{prop}
\begin{proof}
We refer to Chapter 3 of \cite{bcg2} for the same proof in the rank one case.
There are only very few modifications to make in our setting.
Let's first reintroduce a family of immersions which satisfy the conditions above. For a real parameter $c>h(g)$, we consider :
\[\Psi_c(x,\theta)=\left(\int_{\tilde{M}}e^{-cd(x,y)} p_0(y^1,\theta^1)\cdots p_0(y^n,\theta^n)dv_g(y)\right)^{1/2}.\]
The condition on $c$ ensures that the integral converges. Indeed for uniform lattices volume entropy and critical exponent are the same. Then we define an element of $\mathcal{N}$ by
\[\Phi_c(x,\theta)=\frac{\Psi_c(x,\theta)}{\left(\int_{\mathbb{T}^n}\Psi_c^2(x,\theta)d\theta\right)^{1/2}}.\]
We just have replaced the boundary sphere by the Furstenberg boundary. We can now perform the very same computation as \cite{bcg2} (p.742 for a proof of the Lipschitz regularity and p.746 for the volume computation). We get the required estimate for the spherical volume.
\end{proof}

It remains to check now, if $M$ is a compact quotient of $\left(\mathbb{H}^2\right)^n$, that
\[\SphereVol(M) = \left(\frac{h(g_0)^2}{8n}\right)^{n} \Vol (M,g_0).\]
In fact, we can find an immersion $\Phi$ of $\mathcal{N}$ which has precisely the needed volume.

\begin{lem}\label{poisson}
Let $M$ be a compact quotient of $\left(\mathbb{H}^2\right)^n$. Let $\Phi_0 : \left(\mathbb{H}^2\right)^n \rightarrow \Le^2(\mathbb{T}^n)$ be defined by
\[\Phi_0(x^1,\cdots,x^n,\theta^1,\cdots,\theta^n)=\prod_{i=1}^n\sqrt{\frac{1-\abs{x^i}^2}{\abs{x^i-e^{2i\pi\theta^i}}^2}}.\]
Then
\begin{enumerate}
\item $\Phi_0$ is a smooth embedding, it belongs to $\mathcal{N}$ and
\[\Vol(\Phi_0)=\left(\frac{h(g_0)^2}{8n}\right)^n\Vol(g_0)=\left(\frac{1}{8}\right)^n\Vol(g_0).\]
\item The tangent space at the basepoint $\mathds{1}\in\mathcal{S}^\infty$ of the image of $\Phi_0$, $T_{\mathds{1}}\Phi_0(\left(\mathbb{H}^2\right)^n)$ is generated by the $2n$ functions
\[\fonction{f_i}{\mathbb{T}^n}{\R}{\theta=(\theta^1,\cdots,\theta^n)}{\sqrt{2}\cos\theta^i}\]
and
\[\fonction{f_{i+1}}{\mathbb{T}^n}{\R}{\theta=(\theta^1,\cdots,\theta^n)}{\sqrt{2}\sin\theta^i},\]
for $i=1,\cdots,n$.
\end{enumerate}
\end{lem}

\begin{proof}
It is enough to handle the same situation with only one factor $\mathbb{H}^2$, the situation appearing completely as a product. The family of measures $(\nu_z)_{z\in\mathbb{H}^2}$ on $\mathbb{S}^1$ which are in the Lebesgue class and satisfy
\[\frac{d\nu_z}{d\theta}=\frac{1-\abs{z}^2}{\abs{z-e^{2i\pi\theta}}^2} \]
are in fact the so-called Patterson-Sullivan measures of the hyperbolic plane which were constructed in \cite{patterson}. We refer to this original paper for the equivariance relation (remark this family of measures is even SL$_2(\R)$-equivariant).\\

Then it is enough to show that $\Phi_0$ is an immersion at the basepoint $o$, SL$_2(\R)$ acting (transitively on $\mathbb{H}^2$) by diffeomorphisms. The differential of $\Phi_0$ is easy to compute and we get the basis we claimed for the tangent.

This shows in particular the point 2 of the above Lemma which shall be used later. As the $2$ functions $\cos$ and $\sin$ are a free family in $\Le^2(\mathbb{S}^1)$, we obtain that $\Phi_0$ is an immersion at $o$.\\

Finally the volume of $\Phi_0$ has been computed in \cite{bcg2} at page 744.
\end{proof}

\subsection{Calibration theory}\label{calibration}

In order to show that the Spherical Volume is achieved by the map $\Phi_0$, we use a classical method of calibration, following \cite{bcg2} chapter 4. Let us make a brief review on how we implement this method.\\

Let $\Omega$ be a differential $2n$-form which is $\Gamma$-invariant.

\begin{defi}
\begin{enumerate}
\item The comass of $\Omega$ is the quantity
\[\comass(\Omega)=\sup \abs{\Omega_{\varphi}(f_1,\cdots,f_{2n})},\]
where the supremum is taken over all functions $\varphi\in\mathcal{S}^\infty$ and every orthonormal family $(f_1,\cdots,f_{2n})$ where each $f_i$ belongs to $T_{\varphi}\mathcal{S}^\infty$.\
\item One says that the differential form $\Omega$ calibrates some immersion $\Phi_0\in\mathcal{N}$ if
\begin{enumerate}
\item The form $\Omega$ is closed,\
\item its comass is finite and nonzero and \
\item when we restrict $\Omega$ to orthonormal families, it is maximal on the tangent space $T\Phi_0(\left(\mathbb{H}^2\right)^n)$, that is,
\[\frac{\abs{\Omega_{\Phi_0(x)}(d_x\Phi_0(u_1), \cdots,d_x\Phi_0(u_{2n}))}}{\norm{d_x\Phi_0(u_1)\wedge \cdots\wedge d_x\Phi_0(u_{2n})}} =\comass(\Omega)\]
for every $x\in\left(\mathbb{H}^2\right)^n$ and every orthonormal family $(u_1,\cdots,u_{2n})$
\end{enumerate}
\end{enumerate}
\end{defi}

Here is the way we shall exploit a calibrating differential form. The following proposition follows readily from Stokes theorem. Here, we emphasize the fact that we use in a decisive way the compactness hypothesis for $M$ (see \cite{stormnonuniform}).

\begin{prop}[\cite{bcg2} proposition 4.3 p.748]
Assume there exists a differential $2n$ form which calibrates an immersion $\Phi_0\in\mathcal{N}$. Then
\[\SphereVol(M)=\Vol (\Phi_0). \]
\end{prop}

From now on it remains to find such a calibrating form for the Poisson kernel.

\section{A calibrating form for \texorpdfstring{$\left(\mathbb{H}^2\right)^n$}{h22}}\label{proof}

\subsection{Definition of the form}\label{form}

We will denote by $e$ the Euler class of the circle,
\[
   \left \{
   \begin{array}{r c l}
      e(\theta_0,\theta_1,\theta_2)  & =  & 1 \mbox{ if the points are cyclically ordered on $\mathbb{S}^1$  } \\
      e(\theta_0,\theta_1,\theta_2)  & =  & -1 \mbox{ if not} \\
   \end{array}
   \right .
\]
Then we consider the application $C : \left(\mathbb{T}^n\right)^{2n+1}\rightarrow \R$ given by the following formula,
\[C(\theta_0,\cdots,\theta_{2n})=\frac{1}{(2n+1)!} \sum_{\sigma\in\mathfrak{S}_{2n+1}}\sign (\sigma) \prod_{i=1}^ne(\theta^i_{\sigma(2i-2)},\theta^i_{\sigma(2i-1)},\theta^i_{\sigma(2i)}). \]
For example if $n=1$ then $C=e$, the Euler class and if $n=2$, this is the application used in \cite{bucherh22}. This map $C$ could be seen as the alternation of the cup product of $n$ Euler classes. In particular $C$ is alternate, that is,
\[C(\theta_{\sigma(0)},\cdots,\theta_{\sigma(2n)})=\sign (\sigma) C(\theta_0,\cdots,\theta_{2n}).\]
Then the following formula defines a differential $2n$-form on $\mathcal{S}^\infty$:
\[\Omega_{\varphi}(f_1,\cdots,f_{2n})=\int_{\left(\mathbb{T}^n\right)^{2n+1}} C(\theta_0,\cdots,\theta_{2n})\varphi^2(\theta_0) \varphi f(\theta_1)\cdots \varphi f (\theta_{2n}) d\theta_0\cdots d\theta_{2n}.
\]
We conclude this paragraph by stating some properties of $C$ that we shall use later on.

\begin{prop}\label{invariancecocycle}
Let $G$ be the group $\left(\Diff^+(\mathbb{S}^1)\right)^n$ embedded diagonally in $\Diff^+(\mathbb{T}^n)$. The map $C$ is $G$-invariant, that is
\[\forall g\in G,\;\; C(g\theta_0,\cdots,g\theta_{2n})=C(\theta_0,\cdots,\theta_{2n}). \]
\end{prop}

\begin{proof}
The group $\Diff^+(\mathbb{S}^1)$ preserves the cyclic order on $\mathbb{S}^1$ then, for $\gamma\in \Diff^+(\mathbb{S}^1)$,
\[e(\gamma\theta_0,\gamma\theta_1,\gamma\theta_2)=e(\theta_0,\theta_1,\theta_2).\]
Taking $g=(\gamma^1,\cdots,\gamma^n)\in G$, one has
\begin{eqnarray*}
C(g\theta_0,\cdots,g\theta_{2n}) & = & \frac{1}{(2n+1)!} \sum_{\sigma\in\mathfrak{S}_{2n+1}}\sign (\sigma) \prod_{i=1}^n e(\gamma^i\theta^i_{\sigma(2i-2)},\gamma^i\theta^i_{\sigma(2i-1)},\gamma^i\theta^i_{\sigma(2i)})\\
 & = & \frac{1}{(2n+1)!} \sum_{\sigma\in\mathfrak{S}_{2n+1}}\sign (\sigma) \prod_{i=1}^ne(\theta^i_{\sigma(2i-2)},\theta^i_{\sigma(2i-1)},\theta^i_{\sigma(2i)})\\
 & = & C(\theta_0,\cdots,\theta_{2n}).
\end{eqnarray*}
\end{proof}

\begin{prop}\label{closedcocycle}
The map $C$ is closed as a combinatorial cochain, that is
\[\sum_{i=0}^{2n+1}(-1)^i C(\theta_0,\cdots,\hat{\theta_i},\cdots,\theta_{2n+1})=0. \]
\end{prop}

\begin{proof}
We refer to \cite{bucherpolygons} p.331.
\end{proof}

\subsection{An invariance relation satisfied by \texorpdfstring{$\Omega$}{w}}\label{invarianceform}

Let us show now that the form $\Omega$ is invariant under the action of the group $G=\left(\Diff^+(\mathbb{S}^1)\right)^n$. This group contains $\left(\PSL_2(\R)\right)^n$ and extends its action on $\Le^2(\mathbb{T}^n)$ by the change of variables formula,
\[(gf)(\theta)=\sqrt{\Jac g^{-1}(\theta)}f\circ g^{-1}(\theta). \]
The action being unitary, we let act $G$ on the unit sphere $\mathcal{S}^\infty$ by restriction and on the tangent space of this sphere. Then we have, for $g\in G$,
\begin{eqnarray*}
(g^*\Omega)_{\varphi}(f_1,\cdots,f_{2n}) &  = &  \Omega_{g\varphi}(gf_1,\cdots,gf_{2n})\\
& =  & \int_{\left(\mathbb{T}^n\right)^{2n+1}} C(\theta_0,\cdots,\theta_{2n})(g\varphi)^2(\theta_0)\prod_{i=1}^{2n} (g\varphi) (gf_i)(\theta_i)d\theta_{i}
\\
& =  & \int_{\left(\mathbb{T}^n\right)^{2n+1}} C(\theta_0,\cdots,\theta_{2n})\Jac g^{-1}(\theta_0)\varphi^2(\theta_0) \prod_{i=1}^{2n}\Jac g^{-1}(\theta_i)\varphi f_i(\theta_i)d\theta_{i}.
\end{eqnarray*}
We now perform the change of variables formula $\theta_i'=g^{-1}(\theta_i)$. We get
\[(g^*\Omega)_{\varphi}(f_1,\cdots,f_{2n}) = \int_{\left(\mathbb{T}^n\right)^{2n+1}} C(g\theta_0,\cdots,g\theta_{2n})\varphi^2(\theta_0) \varphi f(\theta_1)\cdots \varphi f (\theta_{2n}) d\theta_0\cdots d\theta_{2n}.\]
The conclusion now follows from proposition \ref{invariancecocycle}.

\subsection{Closure of \texorpdfstring{$\Omega$}{w}}\label{closure}

In order to use the method of calibration, we have to deal with a closed form. This is the aim of this paragraph.

\begin{prop}
The differential $2n$-form $\Omega$ is closed.
\end{prop}

\begin{proof}

To differentiate $\Omega$, it is easier to have an expression on a space of measures instead of the unit $\Le^2$-sphere. Let us begin by a quick review on the structure of the space of measures we will deal with.\\

Let $\ronm$ be the Banach vector space dual to $\mathcal{C}^0(\mathbb{T}^n)$, the Banach space of continuous functions on $\mathbb{T}^n$. The space $\ronm$ is also the space of Radon measures on $\mathbb{T}^n$. We consider the affine space
\[\ronm_1=\set{\mu\in\ronm |  \mu(\mathbb{T}^n)=1}.\]
We think as this affine space as a infinite dimensional manifold shaped on a Banach space. The tangent space in each point is the vector space of zero-mass measures.\\

The differential form $\Omega$ comes from a form on $\ronm_1$ pulled-back by the smooth map
\[\fonction{\mathcal{A}}{\mathcal{S}^\infty}{\ronm_1}{\varphi}{B\mapsto\int_B\varphi^2(\theta)d\theta}. \]
The image of a function $\varphi$ is the measure with density $\varphi^2$ with respect to the Lebesgue measure.
Let us define a differential  $2n$-form $\Omega'$ on $\ronm_1$,
\[
\Omega'_{\mu}(\alpha_1,\cdots,\alpha_{2n})  = \int_{\left(\mathbb{T}^n\right)^{2n+1}}C(\theta_0,\cdots,\theta_{2n})d\mu(\theta_0)d\alpha_1(\theta_1)\cdots d\alpha_{2n}(\theta_{2n}).
\]
Clearly
\[\Omega=2^{2n}\mathcal{A}^*\Omega'.\]
Then it is enough to show that $\Omega'$ is closed. But now, $\Omega'$ is a linear map with respect to $\mu$. Then, we have (see \cite{lang} p.84 for the differential formula of a form on a Banach manifold)
\[d\Omega_{\mu}'(\alpha_0,\cdots,\alpha_{2n})=\sum_{i=0}^{2n}(-1)^i\partial_{\mu}\Omega_{\mu}(\alpha_0,\cdots,\hat{\alpha_i},\cdots,\alpha_{2n})\cdot\alpha_i .\]
By linearity this expression is also
\[d\Omega_{\mu}'(\alpha_0,\cdots,\alpha_{2n})=\sum_{i=0}^{2n}(-1)^i\Omega_{\alpha_i}(\alpha_0,\cdots,\hat{\alpha_i},\cdots,\alpha_{2n}).\]
With skew-symmetry, this is again
\[d\Omega_{\mu}'(\alpha_0,\cdots,\alpha_{2n}) = (2n+1)\int_{\left(\mathbb{T}^n\right)^{2n+1}} C(\theta_0,\cdots,\theta_{2n})d\alpha_0(\theta_0)\cdots d\alpha_{2n}(\theta_{2n}).\]
Each measure has a vanishing total mass. Let us show that the above expression is in fact
\[d\Omega_{\mu}'(\alpha_0,\cdots,\alpha_{2n}) = -(2n+1)\int_{\left(\mathbb{T}^n\right)^{2n+2}}\pa C(\theta_0,\cdots,\theta_{2n+1})d\alpha_0(\theta_0)\cdots d\alpha_{2n}(\theta_{2n})d\mu(\theta_{2n+1}).\]
\label{vanishing}
Let us recall the definition of the combinatorial boundary of a cochain,
\[\pa C(\theta_0,\cdots,\theta_{2n+1})=\sum_{i=0}^{2n+1} (-1)^i C(\theta_0,\cdots,\hat{\theta_i},\cdots,\theta_{2n+1}). \]
First we have,
\[\int_{\left(\mathbb{T}^n\right)^{2n+1}}C(\theta_0,\cdots,\theta_{2n})d\alpha_0(\theta_0)\cdots d\alpha_{2n}(\theta_{2n})\]
\[= \int_{\left(\mathbb{T}^n\right)^{2n+2}} C(\theta_0,\cdots,\theta_{2n})d\alpha_0(\theta_0)\cdots d\alpha_{2n}(\theta_{2n})d\mu(\theta_{2n+1})\]
because $\mu(\mathbb{T}^n)=1$.
We recognize the opposite of the last term in the expression of $\pa C$ (corresponding to the index $i=2n+1$).
We now prove that the other terms in the expression of $\pa C$ vanish. Each of this remaining terms are in the form of
\[\int_{\left(\mathbb{T}^n\right)^{2n+2}}C(\theta_0,\cdots,\hat{\theta_i},\cdots,\theta_{2n+1})d\alpha_0(\theta_0)\cdots d\alpha_{2n}(\theta_{2n})d\mu(\theta_{2n+1})\]
with $i\neq 2n+1$. Integrating with respect to the variable $\theta_i$ which is not involved in the cocycle, make appear a multiplicative term, the total mass of one of a tangent measure, supposed to be zero.\\

So finally the closure of $C$ as a combinatorial cocycle (see proposition \ref{closedcocycle}) proves the result.
\end{proof}

\subsection{The calibrating inequality}\label{inequality}

Let us first remark that we only have to establish the inequality at the basepoint $\varphi=\mathds{1}$. Indeed one can show exactly as in paragraph \ref{closure} that
\[\int_{\left(\mathbb{T}^n\right)^{2n+1}} C(g\theta_0,\cdots,g\theta_{2n})\left(\varphi^2(\theta_0)-1\right) \varphi f(\theta_1)\cdots \varphi f (\theta_{2n}) d\theta_0\cdots d\theta_{2n}=0.\]
Replacing the functions $f_i\in T_{\varphi}\mathcal{S}^\infty$ by the functions $\varphi f_i\in T_{\mathds{1}}\mathcal{S}^\infty$, we can make the assumption $\varphi=\mathds{1}$.\\

Besson, Courtois and Gallot studied the $2$-dimensional case in \cite{bcg2} chapter 6. They used the Euler class to build a differential form
\[\omega_{\varphi}(f_1,f_2)= \int_{\left(\mathbb{S}^1\right)^3}e(\theta_0,\theta_1,\theta_2)\varphi^2(\theta_0) \varphi f_1(\theta_1)\varphi f_2(\theta_2)d\theta_0 d\theta_1 d\theta_2.\]
Here is what they proved

\begin{prop}[\cite{bcg2} chapter 6]\label{bcgcalibration}
\begin{enumerate}
\item One may use the following alternative expression for $\omega$,
\[\omega_{\varphi}(f_1,f_2)=2\int_{\mathbb{S}^1}F_1dF_2,\]
where $F_1$ and $F_2$ are the primitives with vanishing integrals for $\varphi f_1$ and $\varphi f_2$.\
\item The comass of $\omega$ equals $\frac{1}{\pi}$.
\item Moreover $\omega$ is a calibrating form for the Poisson kernel, that is, at the basepoint $\mathds{1}$, $\omega$ is maximal over orthonormal families $(f_1,f_2)$ if
\[f_1(\theta)=\sqrt{2}\cos \theta \;\;\mbox{ and }\;\; f_2(\theta)=\sqrt{2}\sin \theta.\]
\end{enumerate}
\end{prop}

Back to the $2n$-dimensional situation, we first show that the differential form $\Omega$ is nonzero, evaluating it on $T_{\mathds{1}}\Phi_0(\left(\mathbb{H}^2\right)^n)$. Remember the family of functions given by
\[\fonction{f_i}{\mathbb{T}^n}{\R}{\theta=(\theta^1,\cdots,\theta^n)}{\sqrt{2}\cos\theta^i}\]
and
\[\fonction{f_{i+1}}{\mathbb{T}^n}{\R}{\theta=(\theta^1,\cdots,\theta^n)}{\sqrt{2}\sin\theta^i},\]
for $i=1,\cdots,n$, is an orthonormal basis of $T_{\mathds{1}}\Phi_0(\left(\mathbb{H}^2\right)^n)$.

\begin{lem}\label{comass}
We have the relation
\[\Omega_{\mathds{1}}(f_1,f_2,\cdots,f_{2n})=\frac{2^n}{\pi^n(2n)!}.\]
\end{lem}

\begin{proof}
Let us look for permutations $\sigma\in\mathfrak{S}_{2n+1}$ for which the corresponding term
\[2^n \sign(\sigma)\int_{\left(\mathbb{T}^n\right)^{2n+1}}\prod_{i=1}^n\left(e(\theta^i_{\sigma(2i-2)},\theta^i_{\sigma(2i-1)},\theta^i_{\sigma(2i)}) \cos \theta_{2i-1}^i \sin\theta_{2i-2}^i\right) d\theta_0\cdots d\theta_{2n}\]
is nonzero. In order to get such a permutation, the variables $\theta_{2i-1}^i$ and $\theta_{2i}^i$ which are involved in the expression
\[  \cos \theta_1^1 \sin\theta_2^1 \cdots \cos \theta_{2n-1}^n \sin\theta_{2n}^n\]
have to appear also in the expression
\[\prod_{i=1}^ne(\theta^i_{\sigma(2i-2)},\theta^i_{\sigma(2i-1)},\theta^i_{\sigma(2i)}).\]
If not we integrate with respect to one of the variables $\theta_{2i-1}^i$ or $\theta_{2i}^i$ missing in the term of the cocycle and we get a multiplicative vanishing term which is the integral of one of the functions $f_i$.
Hence, we need that, for each $i$,
\[ \set{\theta_{2i-1},\theta_{2i}}\subset \set{\theta_{\sigma(2i-2)},\theta_{\sigma(2i-1)},\theta_{\sigma(2i)}}. \]
Let us now compute the number of suitable permutations. The first image $\sigma(0)$ could be any of the $2n+1$ elements. Let us say $\sigma(0)$ falls in the $i$\textsuperscript{th} Euler class, $\sigma(0)\in \set{2i-2,2i-1,2i}$ (e.g $\sigma(0)=2i-2$, other cases are similar). From the condition above, it remains only two possibilities for $\sigma(2i-1)$ and $\sigma(2i)$, namely,
\[\sigma(2i-1)=\sigma(2i-1) \;\;\mbox{ and }\;\; \sigma(2i)=\sigma(2i)\]
or
\[\sigma(2i-1)=\sigma(2i) \;\;\mbox{ and }\;\; \sigma(2i)=\sigma(2i-1).\]
Once we fill the $i$\textsuperscript{th} Euler class, the adjacent one(s) already have one element imposed (remember $\theta_{2i}$ appears in both the $i$\textsuperscript{th} Euler class and the $(i+1)$\textsuperscript{th} Euler class). Hence we have again two possibilities for the two remaining arguments and inductively for each Euler class.

Finally we find
\[2^n(2n+1)\]
suitable permutations.\\

We now play with the skew-symmetry relation to show that, for each suitable permutation $\sigma$, the corresponding term
\[2^n \sign(\sigma)\int_{\left(\mathbb{T}^n\right)^{2n+1}}\prod_{i=1}^n\left(e(\theta^i_{\sigma(2i-2)},\theta^i_{\sigma(2i-1)},\theta^i_{\sigma(2i)}) \cos \theta_{2i-1}^i \sin\theta_{2i-2}^i\right) d\theta_0\cdots d\theta_{2n}\]
equals $\frac{1}{\pi^n}$. From now on, we assume that $\sigma$ is some suitable permutation characterized above.

The variable $\theta_0$ appears in one or two Euler classes, let us say for example that
\[\sigma(0)=2i\]
for some $i<n$ (other cases are simpler). Then the variable $\theta_0$ appears in both the $i$\textsuperscript{th} and the $(i+1)$\textsuperscript{th} Euler class. We compute the term
\[T_{\sigma}:= \sign(\sigma)\int_{\left(\mathbb{T}^n\right)^{2n+1}}\prod_{i=1}^n\left(e(\theta^i_{\sigma(2i-2)},\theta^i_{\sigma(2i-1)},\theta^i_{\sigma(2i)}) \cos \theta_{2i-1}^i \sin\theta_{2i-2}^i\right) d\theta_0\cdots d\theta_{2n}\]
by first performing the change of variables
\[\theta_{\sigma(2i-2)}=\theta'_1,\;\; \theta_{\sigma(2i-1)}=\theta'_2,\;\; \theta_{\sigma(2i+1)}=\theta'_3, \;\; \theta_{\sigma(2i+2)}=\theta'_4 \;\; \mbox{ and } \;\; \theta_k=\theta'_k,\]
for $k\neq 2i-2, 2i-1, 2i+1, 2i+2$.
Since the Jacobian of this change of variable is $1$, we do not change the value of $T_{\sigma}$. Remember we took $\sigma$ such that
\[\set{\sigma(2i-2),\sigma(2i-1)}=\set{2i-2,2i-1}\; \mbox{ and } \set{\sigma(2i+1),\sigma(2i+2)}=\set{2i+1,2i+2}.\]
So we can assume that $\theta_0$ appears in the first two Euler classes.

Now the support disjoint cycles of the permutation $\sigma$ are all contained in sets of the form
\[\set{2i-2,2i-1,2i}.\]
Hence, when one reorders the variables in each Euler class such that
$\theta_{\sigma(2i-1)}$ sits in the $(2i-1)$\textsuperscript{th} place and $\theta_{\sigma(2i)}$ sits in the $(2i)$\textsuperscript{th} place, the sign by which we changed $T_{\sigma}$ is precisely $\sign(\sigma)$.

Therefore, because the first variables in each Euler class don't play any role, each term $T_{\sigma}$ equals $T_{id}$. The computation of $T_{id}$ is easy because the situation appears now as a product and one can use Lemma \ref{bcgcalibration}. Indeed,
\begin{eqnarray*}
T_{id} & = & 2^n \prod_{i=1}^n\int_{\left(\mathbb{S}^1\right)^3}e(\theta^i_{2i-2},\theta^i_{2i-1},\theta_{2i}^i) \cos \theta_{2i-1}^i \sin_{2i}^id\theta_{2i-2}^id\theta_{2i-1}^id\theta_{2i}^i\\
 & = & \prod_{i=1}^n \comass(\omega)\\
 & = & \frac{1}{\pi^n}
\end{eqnarray*}
\end{proof}

We now proceed to the calibrating inequality. We introduce the $(2n-1)$-cocycle,
\[D(\theta_1,\cdots,\theta_{2n})=\int_{\mathbb{T}^n}C(\theta_0,\cdots,\theta_{2n})d\theta_0.\]
The reasoning involves the Fourier coefficient of $D$. A function $f\in T_{\mathds{1}}\mathcal{S}^\infty$ is written as
\[f(\theta^1,\cdots,\theta^n)=\sum_{k=(k^1,\cdots,k^n)\in I}a_k\cos (k^1\theta^1+\cdots+k^n\theta^n) + b_k\sin (k^1\theta^1+\cdots+k^n\theta^n),\]
where $\theta=(\theta^1,\cdots,\theta^n)\in\mathbb{T}^n$ and $k\in I$, the subset of $\mathbb{Z}^n$ given by
\[I=\N\backslash\set{0}\times \Z^{n-1} \cup \set{0}\times \N\backslash \set{0}\times\Z^{n-2}\cup \cdots \cup \set{0,\cdots,0}\times\N\backslash \set{0}.\]
The convergence of the sum has to be understood for the $\Le^2$ topology.

We denote by $\gamma : \mathbb{S}^1 \rightarrow \R$ a function which can be either $\sqrt{2}\cos$ or $\sqrt{2}\sin$. For some $K=(k^1,\cdots,k^n)\in I$ we denote by $\gamma_K : \mathbb{T}^n\rightarrow \R$ a function which can be either
\[\theta=(\theta^1,\cdots,\theta^n)\longmapsto \sqrt{2}\cos(k^1\theta^1+\cdots +k^n\theta^n)\]
or
\[\theta=(\theta^1,\cdots,\theta^n)\longmapsto \sqrt{2}\sin(k^1\theta^1+\cdots +k^n\theta^n)\]

\begin{defi}[Fourier transform and Fourier coefficients]
Let $F$ be a map from $\left(\mathbb{T}^n\right)^{2n}$ to $\R$. The Fourier transform of $F$ is the map
\[\fonction{\hat{F}}{I^{2n}}{\R^{2^{2n}}}{(K_1,\cdots,K_{2n})}{\int_{\left(\mathbb{T}^n\right)^2n} D(\theta_1,\cdots,\theta_{2n}\gamma_{K_1}(\theta_1)\cdots \gamma_{K_{2n}}(\theta_{2n})d\theta_1\cdots d\theta_{2n}}.\]
We called $\hat{F}(K_1,\cdots,K_{2n})$ a Fourier coefficient. The $2^{2n}$ coordinates of a Fourier coefficient correspond to the possible choices of functions $\gamma$.
\end{defi}

\begin{rmq}
Fourier coefficients of the cocycle $D$ contain values of $\Omega$. In particular
\[\hat{D}\begin{pmatrix}
1 & 1 & 0 & 0 & \cdots & 0\\
0 & 0 & 1 & 1 & \cdots & 0\\
\vdots & \vdots & \vdots & \vdots & \vdots & \vdots\\
0 & 0 & &\cdots & 1 & 1
\end{pmatrix}\]
contains the value of $\Omega$ on the orthonormal tangent of the image of the Poisson kernel. That is why we investigate the combinatorial properties of $\hat{D}$
\end{rmq}

Here is a characterization of the nonzero Fourier coefficients of $D$

\begin{lem}\label{nonzerocoeff}
\begin{enumerate}
\item The Fourier coefficient $\hat{D}(K_1,\cdots,K_{2n})$ is nonzero if and only if there exist some nonzero integers $k^1,\cdots,k^n$ such that
\[(K_1,\cdots,K_{2n})=\begin{pmatrix}
k^1 & k^1 & 0 & 0 & \cdots & 0\\
0 & 0 & k^2 & k^2 & \cdots & 0\\
\vdots & \vdots & \vdots & \vdots & \vdots & \vdots\\
0 & 0 & &\cdots & k^n & k^n
\end{pmatrix}\]
up to permutation of lines and rows. Moreover if $(K_1,\cdots,K_{2n})$ is as above, the nonzero coordinates of $\hat{D}(K_1,\cdots,K_{2n})$ (among the $2^{2n}$ ones) is constructed by taking
\[\gamma_{K_{2i-1}}(\theta)=\sqrt{2}\cos(k^i\theta^1)\; \mbox{and} \; \gamma_{K_{2i}}=\sqrt{2}\sin(k^i\theta^n)\]
for all $i$ and up to permutations of the functions.\
\item A nonzero coordinate of some Fourier coefficient, that is,
\[2^{n}\int_{\left(\mathbb{T}^n \right)^{2n+1}}C(\theta_0,\cdots,\theta_{2n})\left(\prod_{i=1}^n\cos(k^i\theta_{2i-1}^i) \sin(k^i\theta_{2i}^i)\right) d\theta_0\cdots d\theta_{2n} \]
equals
\[\frac{2^n}{(2n)!\pi^n\prod_{i=1}^nk^i}.\]
In particular, any nonzero value is smaller than the value of $\Omega$ on the tangent of the image of the Poisson kernel.
\end{enumerate}
\end{lem}

\begin{proof}
Take some functions $\gamma_{K_1},\cdots,\gamma_{K_{2n}}$. One may use the addition formulas for the trigonometric functions $\gamma_{K_i}$. Hence $\hat{D}(K_1,\cdots,K_{2n})$ is a sum of terms of the form
\[\int_{\left(\mathbb{S}^1 \right)^{2n^2}} \prod_{i=1}^ne(\theta^i_{\sigma(2i-2)},\theta^i_{\sigma(2i-1)},\theta^i_{\sigma(2i)}) \prod_{j=1}^{2n}\prod_{l=1}^n \gamma(k_j^l\theta_j^l)d\theta_j^l \]
(if some integer $k_j^l=0$, we set $\gamma(k_j^l\theta_j^l)=1$).
We show that if $(K_1,\cdots,K_{2n})$ is not as in the statement of the Lemma, every such term vanishes.
In order to do that, we will use several times the arguments below.
\begin{enumerate}
\item {\bf Argument 1} : If there is some variable $\theta_j^l\in\mathbb{S}^1$ which is involved in some function $\gamma$ and which is not involved in the term of the cocycle, that is
\[\theta_j^l\notin \bigcup_{i=1}^n\set{\theta^i_{\sigma(2i-2)},\theta^i_{\sigma(2i-1)},\theta^i_{\sigma(2i)}}, \]
then we integrate with respect to this variable $\theta_j^l$ and we get zero as a multiplicative term because the integral of $\theta_j^l\mapsto \gamma(k_j^l\theta_j^l)$ vanishes. \
\item {\bf Argument 2} : One can also show exactly as in paragraph \ref{closure} that
\[\int_{\left(\mathbb{S}^1\right)^3}e(\theta_0,\theta_1,\theta_2)\gamma(k_0\theta_0)\gamma(k_1\theta_1) \gamma(k_2\theta_2)d\theta_0 d\theta_1 d\theta_2=0.\]
Indeed the Euler class $e$ is combinatorially closed.\
\item {\bf Argument 3} : An expression of the form
\[\int_{\left(\mathbb{S}^1 \right)^3}e(\theta_0,\theta_1,\theta_2)\gamma(\theta_0)d\theta_0 d\theta_1 d\theta_2 \]
must vanish. Indeed for a fixed $\theta_0$,
\[\int_{\left(\mathbb{S}^1 \right)^2}e(\theta_0,\theta_1,\theta_2)d\theta_1 d\theta_2=0. \]
\end{enumerate}

In each term of the cocycle there is $3n$ variables involved. Then we can't have more than $3n$ nonzero integers (otherwise we use argument 1 above). We can't have also more than 3 nonzero integers in each row. Indeed,
\begin{itemize}
\item If there is 4 or more, there is at least one variable involved in the functions and not in the term of the cocycle. We use argument 1.
\item If there is 3, there is again 2 possibilities. Either the variables involved in the functions and in the cocycle are not the same and we use argument 1; either they are the same and after we integrate with respect to this 3 variable, we use argument 2.
\end{itemize}

Moreover we can't have more than 2 nonzero integers in each line. Indeed there is at most one repetition of the variables in the different Euler classes. In fact there is exactly two nonzero integers in each line. If there is only one, this means that there is only one of three variables in some Euler class matching with a variable in the functions. We integrate with respect to this three variables and we use argument 3.

Hence there is exactly $2n$ nonzero integers in the $2n^2$ coordinates of a nonzero index $(K_1,\cdots,K_{2n})\in I^{2n}$ and exactly two nonzero integers in each line. So there is exactly one nonzero integer in each row.\\

So far we can assume that
\[(K_1,\cdots,K_{2n})=\begin{pmatrix}
k^1 & k^2 & 0 & 0 & \cdots & 0\\
0 & 0 & k^3 & k^4 & \cdots & 0\\
\vdots & \vdots & \vdots & \vdots & \vdots & \vdots\\
0 & 0 & &\cdots & k^{2n-1} & k^{2n}
\end{pmatrix}\]
for $2n$ nonzero integers $k^1,k^2,\cdots,k^{2n}$. Indeed, up to sign, the values of $\hat{D}(K_1,\cdots,K_{2n})$ are left unchanged when we permute the $K_i$'s. One can also reorder the lines by performing some change of variables of the form
\[\theta_j^{'\sigma(l)}=\theta_j^l\]
which only changes the sign. From now on the situation appears completely as a product. One can check that the permutations $\sigma\in\mathfrak{S}_{2n+1}$ giving a nonzero contribution are exactly the ones we characterized in Lemma \ref{comass}. Each term gives the same value. A generic coordinate of $\hat{D}(K_1,\cdots,K_{2n})$ can be written as
\[\frac{2^n}{(2n)!}\prod_{i=1}^n \int_{\left(\mathbb{S}^1\right)^3} e(\theta^i_{2i-2},\theta^i_{2i-1},\theta^i_{2i})\gamma(k^{2i-1}\theta^i_{2i-1})  \gamma(k^{2i}\theta^i_{2i})d\theta^i_{2i-2} d\theta^i_{2i-1} d\theta^i_{2i}.\]
We can now conclude with the assertion 1 of Lemma \ref{bcgcalibration}. Indeed we must choose the integers $k^1,k^2,\cdots,k^{2n}$ and the functions $\gamma_{K_1},\gamma_{K_2},\cdots,\gamma_{K_{2n}}$ such that, for all index $i$, $\gamma_{K_{2i}}$ and the primitive of $\gamma_{K_{2i-1}}$ are not $\Le^2$-orthogonal.\\

We can also perform the exact computation of the nonzero values thanks to the same assertion 1 of Lemma \ref{bcgcalibration}. Indeed for each index $i$,
\[
2\int_{\left(\mathbb{S}^1\right)^3} e(\theta^i_{2i-2},\theta^i_{2i-1},\theta^i_{2i})\cos(k^i\theta^i_{2i-1})  \sin(k^i\theta^i_{2i})d\theta^i_{2i-2} d\theta^i_{2i-1} d\theta^i_{2i} \] 
\[=  4\int_{\mathbb{S}^1}\frac{1}{k^i}\sin^2(k^i\theta^i)d\theta^i
 = \frac{1}{k^i\pi} \]
\end{proof}

We finally deduce the calibrating inequality. Take $2n$ functions $f_1,\cdots,f_{2n}$ with integrals zero, the family $(f_1,\cdots,f_{2n})$ being orthonormal. For each function $f_i$, we write
\[f_i(\theta)=\sum_{K=(k^1,\cdots,k^n)\in I}a_{K,i}\cos(k^1\theta^1+\cdots + k^n\theta^n) + b_{K,i}\sin(k^1\theta^1+\cdots + k^n\theta^n) .\]
and we denote by $c_{K,i}$ some Fourier coefficient of $f_i$ which can be either $a_{K,i}$ or $b_{K,i}$. Beside, let us set the following convention in order to design the coordinates of some Fourier coefficient. We take $C=(c_{K_1,1},\cdots,c_{K_{2n},2n})\in\prod_{i=1}^{2n}\set{a_{K_i,i},b_{K_i,i}}$ and denote by $\hat{D}(K_1,\cdots,K_{2n})^C$ the coordinate of $\hat{D}(K_1,\cdots,K_{2n})$ corresponding to
\[
   \left \{
   \begin{array}{r c l}
      \gamma_{K_i}(\theta)  & =  & \cos(k^1_i\theta^1+\cdots +k^n_i\theta^n) \mbox{ if } c_{K_i,i} = a_{K_i,i}  \\
      \gamma_{K_i}(\theta)  & =  & \sin(k^1_i\theta^1+\cdots +k^n_i\theta^n) \mbox{ if } c_{K_i,i} = b_{K_i,i} \\
   \end{array}
   \right .
\]
Hence
\[\Omega_{\mathds{1}}(f_1,\cdots,f_{2n})=\sum_{\substack{K_1,\cdots,K_{2n}\\ C\in\prod_{i=1}^{2n}\set{a_{K_i,i},b_{K_i,i}}}}\hat{D}(K_1,\cdots,K_{2n})^C c_{K_1,1}\cdots c_{K_{2n},2n}. \]
We already isolated the vanishing Fourier coefficients. Let $B=(e_1,\cdots,e_n)$ be a basis of $\R^n$. For $\rho\in\mathfrak{S}_n$, we denote by $Q(\rho)$ the $n\times n$ matrix in $B$ of the linear map sending $e_i$ to $e_{\rho (i)}$. We have
\[\Omega_{\mathds{1}}(f_1,\cdots,f_{2n})=\sum\hat{D}\left(Q(\rho)\cdot\begin{pmatrix}
k^1 & k^1 & 0 & 0 & \cdots & 0\\
0 & 0 & k^2 & k^2 & \cdots & 0\\
\vdots & \vdots & \vdots & \vdots & \vdots & \vdots\\
0 & 0 & &\cdots & k^n & k^n
\end{pmatrix}\right)\prod_{i=1}^nc_{k^{\rho(i)},2i-1}c'_{k^{\rho(i)},2i}  \]
The sum is taken over all possible choices of nonzero integers $k^1,\cdots,k^n$, all permutations $\rho\in \mathfrak{S}_n$ and over all possibilities of $c_{k^{\rho(i)},2i-1}$ and $c'_{k^{\rho(i)},2i}$ such that, for all $i$,
\[ c_{k^{\rho(i)},2i-1}=a_{k^{\rho(i)},2i-1}\;\mbox{ and }\; c'_{k^{\rho(i)},2i}=b_{k^{\rho(i)},2i}\]
or
\[ c_{k^{\rho(i)},2i-1}=b_{k^{\rho(i)},2i-1}\;\mbox{ and }\; c'_{k^{\rho(i)},2i}=a_{k^{\rho(i)},2i}.\]
We now have
\begin{eqnarray*}
\abs{\Omega_{\varphi}(f_1,\cdots,f_{2n})} & \leqslant & \sum\frac{2^n}{(2n)!\pi^n\prod_{i=1}^n k^i}\prod_{i=1}^n\abs{c_{k^{\rho(i)},2i-1}c'_{k^{\rho(i)},2i}}\\
 & \leqslant & \sum \frac{2^n}{(2n)!\pi^n}\prod_{i=1}^n\abs{c_{k^{\rho(i)},2i-1}c'_{k^{\rho(i)},2i}}.
\end{eqnarray*}
It is a standard fact that, using repetitively the Cauchy-Schwartz inequality, one gets
\[\sum_{\alpha\in A}\prod_{\beta=1}^N x_{\alpha,\beta}\leqslant \prod_{\beta=1}^N \sqrt{\sum_{\alpha\in A}x_{\alpha,\beta}^2}, \]
for some countable set $A$, some integer $N$ and some non-negative numbers $x_{\alpha,\beta}$.
Here we obtain
\begin{eqnarray*}
\abs{\Omega_{\mathds{1}}(f_1,\cdots,f_{2n})} & \leqslant & \frac{2^n}{(2n)!\pi^n}\prod_{i=1}^n\sqrt{\sum\abs{c_{k^{\rho(i)},2i-1}}} \prod_{i=1}^n\sqrt{\sum\abs{c'_{k^{\rho(i)},2i}}} \\
 & \leqslant & \frac{2^n}{(2n)!\pi^n} \prod_{i=1}^{2n}\norm{f_i}_{\Le^2}\\
 & = & \frac{2^n}{(2n)!\pi^n}.\\
\end{eqnarray*}
This inequality is an equality when $(f_1,\cdots,f_{2n})$ generates the tangent to the image of the Poisson kernel. The proof of the main theorem is then complete.

\section{Applications}\label{applications}

We finally look for some consequences, suggested by the last chapter of \cite{bcg2}.

It's possible to extend the main result to the case where $g$ lives on another differentiable manifold related to $M$ by a map of non-zero degree. The result we obtain is an optimal "degree theorem" as in the article \cite{connellfarbdegree} who misses the case we investigate. There is also a similar result in \cite{lohsauer} but with a nonoptimal constant and with an additional hypothesis on $\Ric g$.

We denote again $M=\Gamma\backslash \left(\mathbb{H}^2\right)^n$ a compact quotient of $\left(\mathbb{H}^2\right)^n$ and $g_0$ the sum of hyperbolic metrics in the different factors $\mathbb{H}^2$.

\begin{cor}
Let $Y$ be a differentiable manifold of dimension $2n$ endowed with a Riemannian metric $g$ and let $f$ be a continuous map
\[f : (Y,g) \longrightarrow (M,g_0)\]
Then
\[h(g)^{2n} \Vol(Y,g) \geqslant \abs{\textit{deg }f} h(g_0)^{2n} \Vol(M,g_0)\]
\end{cor}

\begin{proof}
Observe that the inequality is trivial if deg$(f)=0$. So let us assume that deg$(f)$ is non-zero.
First, one can regularize the map $f$ in a homotopic map, still denoted $f$, which is $\mathcal{C}^1$. We call $\tilde{f}$ the map induced by $f$ from $\tilde{Y}$ to $\tilde{M}$. Let us introduce the invariant appropriated to this new situation
\[\SphereVol(f)=\inf{\Vol((U,\Phi^*(\mbox{can})))}\]
where $\Phi$ are Lipschitz continuous equivariant immersions from $\tilde{Y}$ to $\Le^2(\mathbb{T}^n)$. As before, one example is given by the product of Poisson kernels,
\[\Phi_0(y,\theta)=\prod_{i=1}^n\sqrt{p_0(\tilde{f}(y)^i,\theta^i)},\]
where the $\tilde{f}(y)^i$'s are the coordinates in the factors $\mathbb{H}^2$ and the $\theta^i$'s are the coordinate in $\mathbb{T}^n$.
We also consider
\[\Phi_c(y,\theta)=\left(\frac{\int_{\widetilde{Y}}e^{-cd(y,z)}\Phi_0^2(z,\theta)dv_g(z)}{\int_{\mathbb{T}^n}e^{-cd(y,z)}\Phi_0^2(z,\theta)d\theta}\right)^{1/2}.\]
where $d$ is the $g$-distance in $\widetilde{Y}$.
The two arguments above (page \pageref{boundedabove} for the first and section \ref{proof} for the second) give in this context
\begin{enumerate}
\item $\SphereVol\leqslant \left(\frac{h(g)^2}{8n}\right)^n\Vol(Y,g)$ using the computation of $\Vol(\Phi_c)$\
\item The image of $\Phi_0$ is still calibrated because $f$ is surjective and then $\SphereVol=\Vol(\Phi_0)$. Moreover 
\[\Vol(\Phi_0)=\abs{\mbox{deg }f}\left(\frac{h(g_0)^2}{8n}\right)^n\Vol(X,g_0)\]
\end{enumerate}
\end{proof}

We also obtain an estimate for the minimal volume. Let $X$ be a compact manifold. The minimal volume is defined as
\[\MinVol=\inf \set{\Vol(g),\,\abs{K(g)}\leqslant 1}\]
(see \cite{gromovvabc}).

\begin{cor}
Let $M=\Gamma\backslash\left( \mathbb{H}^2\right)^n$ be a compact quotient of $\left( \mathbb{H}^2\right)^n$. Then
\[\MinVol(M)\geqslant\left(\frac{\sqrt{n}}{2n-1}\right)^{2n}\Vol(g_0).\]
\end{cor}

In particular, we reprove in a qualitative way a general theorem of \cite{lafontschmidt} stating that $\MinVol$ is nonzero.

\begin{proof}
We still follow chapter 9 of \cite{bcg2}. Take a metric $g$ on $M$ with $\abs{K(g)}\leqslant 1$. We deduce an equality on the Ricci curvature
\[\Ric (g)\geqslant-(2n-1)g\]
We then apply the Bishop's inequality (\cite{ghl} p.144) comparing volumes of balls for $g$ and volumes of balls in the hyperbolic $2n$-space $\mathbb{H}^{2n}$. Taking $\log$ and making the radius go to infinity, we have
\[h(g)\leqslant 2n-1.\]
We conclude introducing this inequality in
\[\Vol(g)\geqslant\left( \frac{\sqrt{n}}{h(g)}\right)^{2n}\Vol(g_0).\]
\end{proof}

Here is a last consequence in dimension 4. P. Suarez-Serrato in \cite{serrato} classified the 4-dimensional Thurston geometries admitting a metric of minimal normalized volume entropy. The only missing case were the case of quotients of $\mathbb{H}^2\times\mathbb{H}^2$. In dimension 4, there exist 19 geometries admitting a compact quotient (see the list and references in \cite{serrato} p. 366). P. Suarez-Serrato was able to decide among the 18 geometries (all but $\mathbb{H}^2\times\mathbb{H}^2$) which ones admit a metric of volume 1 with minimal volume entropy (theorem A). Hence we obtain

\begin{cor}
The 4-dimensional Thurston geometries admitting a metric with minimal normalized volume entropy are only those of "hyperbolic type"
\[\mathbb{H}_{\mathbb{R}}^4,\mathbb{H}_{\mathbb{C}}^2 \mbox{ and } \mathbb{H}_{\mathbb{R}}^2\times\mathbb{H}_{\mathbb{R}}^2. \]
\end{cor}

\bibliographystyle{alpha}
\bibliography{/home/bibou/Documents/4.Recherche/6.Macros/Bibliographie/biblio}

\def\cprime{$'$} \def\dbar{\leavevmode\hbox to 0pt{\hskip.2ex \accent"16\hss}d}
\begin{thebibliography}{BCG07}

\bibitem[BCG91]{bcg1}
G.~Besson, G.~Courtois, and S.~Gallot.
\newblock Volume et entropie minimale des espaces localement sym\'etriques.
\newblock {\em Invent. Math.}, 103(2):417--445, 1991.

\bibitem[BCG95]{bcg2}
G.~Besson, G.~Courtois, and S.~Gallot.
\newblock Entropies et rigidit\'es des espaces localement sym\'etriques de
  courbure strictement n\'egative.
\newblock {\em Geom. Funct. Anal.}, 5(5):731--799, 1995.

\bibitem[BCG07]{bcgmilnorwood}
G{\'e}rard Besson, Gilles Courtois, and Sylvestre Gallot.
\newblock In\'egalit\'es de {M}ilnor-{W}ood g\'eom\'etriques.
\newblock {\em Comment. Math. Helv.}, 82(4):753--803, 2007.

\bibitem[BK08]{bucherh22}
Michelle Bucher-Karlsson.
\newblock The simplicial volume of closed manifolds covered by
  {$\mathbb{H}^2\times\mathbb{H}^2$}.
\newblock {\em J. Topol.}, 1(3):584--602, 2008.

\bibitem[BK09]{bucherpolygons}
Michelle Bucher-Karlsson.
\newblock On minimal triangulations of products of convex polygons.
\newblock {\em Discrete Comput. Geom.}, 41(2):328--347, 2009.

\bibitem[Bor63]{borel}
Armand Borel.
\newblock Compact {C}lifford-{K}lein forms of symmetric spaces.
\newblock {\em Topology}, 2:111--122, 1963.

\bibitem[CF03a]{connellfarbdegree}
Christopher Connell and Benson Farb.
\newblock The degree theorem in higher rank.
\newblock {\em J. Differential Geom.}, 65(1):19--59, 2003.

\bibitem[CF03b]{connellfarbmer}
Christopher Connell and Benson Farb.
\newblock Minimal entropy rigidity for lattices in products of rank one
  symmetric spaces.
\newblock {\em Comm. Anal. Geom.}, 11(5):1001--1026, 2003.

\bibitem[Ebe96]{eberlein}
Patrick~B. Eberlein.
\newblock {\em Geometry of nonpositively curved manifolds}.
\newblock Chicago Lectures in Mathematics. University of Chicago Press,
  Chicago, IL, 1996.

\bibitem[GHL04]{ghl}
Sylvestre Gallot, Dominique Hulin, and Jacques Lafontaine.
\newblock {\em Riemannian geometry}.
\newblock Universitext. Springer-Verlag, Berlin, third edition, 2004.

\bibitem[GJT98]{gjt}
Yves Guivarc'h, Lizhen Ji, and J.~C. Taylor.
\newblock {\em Compactifications of symmetric spaces}, volume 156 of {\em
  Progress in Mathematics}.
\newblock Birkh\"auser Boston Inc., Boston, MA, 1998.

\bibitem[Gro82]{gromovvabc}
Michael Gromov.
\newblock Volume and bounded cohomology.
\newblock {\em Inst. Hautes \'Etudes Sci. Publ. Math.}, (56):5--99 (1983),
  1982.

\bibitem[Gro83]{gromovfilling}
Mikhael Gromov.
\newblock Filling {R}iemannian manifolds.
\newblock {\em J. Differential Geom.}, 18(1):1--147, 1983.

\bibitem[Kat88]{katokconformal}
Anatole Katok.
\newblock Four applications of conformal equivalence to geometry and dynamics.
\newblock {\em Ergodic Theory Dynam. Systems}, 8$^*$(Charles Conley Memorial
  Issue):139--152, 1988.

\bibitem[Lan62]{lang}
Serge Lang.
\newblock {\em Introduction to differentiable manifolds}.
\newblock Interscience Publishers (a division of John Wiley \& Sons, Inc.), New
  York-London, 1962.

\bibitem[LS06]{lafontschmidt}
Jean-Fran{\c{c}}ois Lafont and Benjamin Schmidt.
\newblock Simplicial volume of closed locally symmetric spaces of non-compact
  type.
\newblock {\em Acta Math.}, 197(1):129--143, 2006.

\bibitem[LS09]{lohsauer}
Clara L{\"o}h and Roman Sauer.
\newblock Degree theorems and {L}ipschitz simplicial volume for nonpositively
  curved manifolds of finite volume.
\newblock {\em J. Topol.}, 2(1):193--225, 2009.

\bibitem[Man79]{manning}
Anthony Manning.
\newblock Topological entropy for geodesic flows.
\newblock {\em Ann. of Math. (2)}, 110(3):567--573, 1979.

\bibitem[Pat76]{patterson}
S.~J. Patterson.
\newblock The limit set of a {F}uchsian group.
\newblock {\em Acta Math.}, 136(3-4):241--273, 1976.

\bibitem[Shi63]{shimizu}
Hideo Shimizu.
\newblock On discontinuous groups operating on the product of the upper half
  planes.
\newblock {\em Ann. of Math. (2)}, 77:33--71, 1963.

\bibitem[SS09]{serrato}
Pablo Su{\'a}rez-Serrato.
\newblock Minimal entropy and geometric decompositions in dimension four.
\newblock {\em Algebr. Geom. Topol.}, 9(1):365--395, 2009.

\bibitem[Sto06]{stormnonuniform}
P.~A. Storm.
\newblock The minimal entropy conjecture for nonuniform rank one lattices.
\newblock {\em Geom. Funct. Anal.}, 16(4):959--980, 2006.

\end{thebibliography}

\end{document}